\documentclass[11pt,twoside]{amsart}

\usepackage{lmodern}

\usepackage[T1]{fontenc}               
\usepackage[utf8]{inputenc}            
\usepackage[english]{babel}            
\usepackage{enumerate}					
\usepackage{enumitem}
\usepackage{scalerel}
\usepackage{amsmath,amssymb,amsthm,amsfonts,amsbsy,latexsym}   
\usepackage{amsrefs}
\usepackage{tikz}
\usetikzlibrary{cd}					
\usepackage{mathtools} 
\usepackage{interval}                
\usepackage{textcomp}
\calclayout
\usepackage{hyperref}                  

\renewcommand{\P}{\mathbb{P}}
\newcommand{\R}{\mathbb{R}}
\newcommand{\Q}{\mathbb{Q}}
\newcommand{\St}{\textup{Stab}^{\dagger}}

\newcommand{\Z}{\mathbb{Z}}

\newcommand{\AutD}{\textup{Aut}(\textup{D}^{\textup{b}}(}
\newcommand{\MYp}{M_{\sigma_{\scaleto{+}{4pt}}}^Y(v)}
\newcommand{\MYm}{M_{\sigma_{\scaleto{-}{4pt}}}^Y(v)}
\newcommand{\MYtp}{M_{\widetilde{\sigma}_{\scaleto{+}{4pt}}}^{\widetilde{Y}}(\pi^*(v))}
\newcommand{\MYtm}{M_{\widetilde{\sigma}_{\scaleto{-}{4pt}}}^{\widetilde{Y}}(\pi^*(v))}
\newcommand{\Db}{\textup{D}^{\textup{b}}}
\newcommand{\Dbi}{\textup{D}_{i^*}^{\textup{b}}}
\newcommand{\Pic}{\textup{Pic}}

\newcommand{\NS}{\textup{NS}}
\newcommand\Wtilde{\stackrel{\sim}{\smash{\sigma}\rule{0pt}{0.3ex}}}
\newcommand{\CH}{\textup{CH}}

\newtheorem{thm}{Theorem}[section]
\newtheorem{prop}[thm]{Proposition}

\newtheorem{lem}[thm]{Lemma}
\newtheorem{cor}[thm]{Corollary}

\numberwithin{equation}{section}

\theoremstyle{definition}
\newtheorem{definition}[thm]{Definition}

\begin{document}
	\title[]{BIRATIONAL GEOMETRY OF MODULI SPACES OF STABLE OBJECTS ON ENRIQUES SURFACES}
\author[Thorsten Beckmann]{Thorsten Beckmann}
\newcommand{\Addresses}{{
		\bigskip
		\footnotesize
		
		\textsc{Mathematisches Institut, Universität Bonn,
			Endenicher Allee 60, 53115 Bonn, Germany}\par\nopagebreak
		\textit{E-mail address}: \texttt{beckmann@math.uni-bonn.de}
}}
\maketitle

\begin{abstract}
	Using wall-crossing for K3 surfaces, we establish birational equivalence of moduli spaces of stable objects on generic Enriques surfaces for different stability conditions. As an application, we prove in the case of a Mukai vector of odd rank that they are birational to Hilbert schemes. The argument makes use of a new Chow-theoretic result, showing that moduli spaces on an Enriques surface give rise to constant cycle subvarieties of the moduli spaces of the covering K3. 
\end{abstract}

\section{Introduction}
\label{sec:in}
\let\thefootnote\relax\footnotetext{2010 \textit{Mathematics Subject Classification}. 14D20 (Primary); 18E30, 14C15, 14E30, 14J28 (Secondary).
	
	\textit{Keywords and phrases}. Bridgeland stability conditions, Constant cycle Lagrangian, Derived category, Moduli spaces of complexes.}
Moduli spaces of stable sheaves on surfaces are much studied objects. As stability depends on the choice of a polarization, it is interesting to study the dependence of the geometry of the moduli spaces on this choice. The introduction of Bridgeland stability conditions \cite{BridgelandSt} prompted new techniques, which can be applied to study this question. In \cite{BayMacMMP,BayMacPrBiGeo}, Bayer and Macrì have analyzed in detail the birational geometry of moduli spaces on a projective K3 surface $X$. In particular, they proved that crossing a wall induces a birational transformation and that every smooth $K$-trivial birational model of a moduli space can be obtained by varying stability conditions in the distinguished connected component $\St(X)$ of the stability manifold discovered by Bridgeland \cite[Def.\ 11.4]{BridgelandK3}. 

The purpose of this paper is to prove analogous results for moduli spaces of stable objects on an Enriques surface $Y$. The main technique is to consider the covering K3 surface $ \widetilde{Y}$ and use the already established results for $\widetilde{Y}$. 

More precisely, let $\pi \colon \widetilde{Y} \rightarrow Y$ be the universal covering map of degree 2. Denote by $M_{\sigma}^Y(v)$ and $M^{\widetilde{Y}}_{\widetilde{\sigma}}(\pi^*(v))$ the moduli space of $\sigma$-stable (respectively $\widetilde{\sigma}$-stable) objects on $Y$ (respectively $\widetilde{Y}$) with Mukai vector $v\in H_{\textup{alg}}^*(Y,\Z)$ (respectively $\pi^*(v)\in H^*_{\textup{alg}}(\widetilde{Y},\mathbb{Z})$). Here, $H_{\textup{alg}}^*(Y,\Z)$ respectively $H^*_{\textup{alg}}(\widetilde{Y},\mathbb{Z})$ denotes the image of the Mukai vector, see Section \ref{sec:pre}. Using the pullback along $\pi$, we get a 2:1 morphism of the moduli space $M_{\sigma}^Y(v)$ onto a Lagrangian subvariety of $M^{\widetilde{Y}}_{\widetilde{\sigma}}(\pi^*(v))$ \cite{KimEn, Nuerpro}. Applying a result by Marian and Zhao \cite{MarianZhao}, we conclude
\begin{prop}[see Proposition \ref{prop:Mod_enr_CCL}]
	Let $v \in H_{\textup{alg}}^*(Y,\Z)$ be a Mukai vector such that $\pi^*(v)\in H^*_{\textup{alg}}(\widetilde{Y},\mathbb{Z})$ is primitive and $\sigma \in \St(Y)$ a generic stability condition. The image of the morphism
	\begin{align*}
	\pi^* \colon M^Y_{\sigma}(v) \rightarrow M^{\widetilde{Y}}_{\widetilde{\sigma}}(\pi^*(v))
	\end{align*}
	is a constant cycle Lagrangian.
\end{prop}

Recall that a subvariety is called constant cycle if all its points become rationally equivalent in the ambient variety. 

It turns out that this cycle-theoretic property is enough to deduce birational equivalence of moduli spaces with different stability conditions. 

\begin{thm}[see Theorem \ref{thm:mod_enr_wall_cross}]
	\label{thm:intro_wall}
	For $Y$ a generic Enriques surface and a Mukai vector $v \in H^*_{\textup{alg}}(Y,\mathbb{Z})$ such that $\pi^*(v) \in H^*_{\textup{alg}}(\widetilde{Y},\mathbb{Z})$ is primitive the moduli spaces $M_{\sigma}(v)$ and $M_{\tau}(v)$ are birationally equivalent, where $\sigma, \tau \in \St(Y)$ are generic stability conditions. 
\end{thm}

More precisely, we show that the birational transformation obtained by Bayer and Macrì \cite[Thm.\ 1.1]{BayMacMMP} for moduli spaces on K3 surfaces restricted to the moduli spaces on Enriques surfaces induces a birational transformation. To prove this, we use that these are constant cycle Lagrangians and, therefore, cannot be contained in the exceptional locus, cf.\ Proposition \ref{prop:prop_cc_bir}. 

This enables us to apply Fourier--Mukai transforms to the moduli space and change the stability condition without changing its birational type. Using this technique, one relates moduli spaces for different Mukai vectors. 

\begin{cor}[see Theorem \ref{prop:Mod_Enr_bir_Hilb}]
	Let $Y$ be an arbitrary Enriques surface and $v$ a primitive Mukai vector of odd rank. Then, for generic $\sigma \in \St(Y)$ the moduli space $M_\sigma(v)$ is birationally equivalent to some $\textup{Hilb}^n(Y)$.
\end{cor}

If the rank is even, we deduce that the moduli spaces are Calabi--Yau manifolds employing results by Saccà \cite{SaccaEnriques}. To obtain the results for not necessarily generic Enriques surfaces one uses deformation theory and stability conditions in families \cite{RelativeBridgelandModuliSpaces}. 

As a final application of the birational equivalence of wall-crossing we consider the natural nef divisor classes $\ell_{\sigma}$ associated to a stability condition $\sigma$ \cite[Sec.\ 4]{BayMacPrBiGeo}. Since the birational transformation of Theorem \ref{thm:intro_wall} can be obtained as the restriction of the birational transformation of the moduli spaces on the K3 surface, we can furthermore prove that these divisors can be extended continuously to a map
\[
\ell \colon \St(Y) \rightarrow \NS(M_{\sigma}(v)),
\]
see Lemma \ref{lem:ell_agree_wall}. 
We conclude the paper by showing that the nef and semiample divisors $\ell_{\sigma_0,\pm} \in \NS (M^Y_{\sigma_{\pm}}(v))$ are big (Proposition \ref{prop:ell_big}) and discuss whether all minimal models of moduli spaces of stable objects on an Enriques surface are again a moduli space for a (possibly) different stability condition.
\subsection*{Relation to other work}
The independent preprint \cite{NuerYoshioka} deals as well with the birational geometry of moduli spaces of stable objects on Enriques surfaces. It is shown in loc.\ cit.\ that on an Enriques surface and for arbitrary Mukai vector $v$ with $v^2>0$ the moduli spaces $M_{\sigma}(v)$ and $M_{\tau}(v)$ are birational, where $\sigma$ and $\tau$ are generic stability conditions. This is established, in analogy to \cite[Thm.\ 5.7]{BayMacMMP}, by classifying all possible birational phenomena that occur on a wall for $v$, see \cite[Thm.\ 5.8]{NuerYoshioka} for the precise result. Although these results are more general than ours and yield a detailed analysis on all possible wall-crossing types, the approach in this paper is of independent interest and enables one to prove the birational equivalence in important cases in only a few pages. Via deformation theory we can obtain most of the results of \cite[Thm.\ 1.2, Thm.\ 1.3]{NuerYoshioka} as well. 
\subsection*{Acknowledgements}
First and foremost, I thank my advisor Daniel Huybrechts for answering my questions and for helpful comments on an earlier version of this paper. Moreover, I am indebted to Tim Bülles for many valuable discussions and Andrey Soldatenkov for his useful suggestions. The paper benefited from fruitful conversations with Arend Bayer, Emanuele Macrì, Gebhard Martin, and Georg Oberdieck. I am grateful for the careful reading and many suggestions on previous versions by the anonymous referees. This work is part of the author's Master's thesis. The author is partially supported by SFB/TR 45 ‘Periods, Moduli Spaces and Arithmetic of Algebraic Varieties' of the DFG (German Research Foundation).

\subsection*{Notations and conventions}
We work over the complex numbers. The bounded derived category of a smooth projective variety $X$ is denoted by $\Db(X)$. Throughout, Chow groups are the groups of cycles modulo rational equivalence. An Enriques surface is called generic if the Picard rank of its covering K3 surface is 10. In Section \ref{sec:gmes}, unless otherwise specified, we only consider generic Enriques surfaces. The moduli spaces on Enriques surfaces under consideration have two components and we frequently omit the determinant in the notation.
	\section{Constant Cycle Subvarieties}
\label{sec:cc} 
Our approach uses the following class of subvarieties.
\begin{definition}
	\label{def:CCV}
	A subvariety $Y\subset X$ is called a \textit{constant cycle subvariety} if all points in $Y$ are rationally equivalent as points in $X$. If, moreover, $X$ is a symplectic variety and $Y$ is a Lagrangian subvariety, then $Y$ is called \textit{constant cycle Lagrangian}.
\end{definition}
Observe that if $X$ is symplectic, then every constant cycle subvariety is isotropic, which follows from Roitman's Theorem \cite[Prop.\ 10.24]{VoisinBookII}. Thus, in the definition one may replace the Lagrangian assumption with a condition on the dimensions $2 \dim(Y)=\dim(X)$. 

We refer to \cite{HuybrechtsCCC,VoisinRemarkandQuestions} for further discussions on constant cycle subvarieties. For our purposes we only need the following properties.
\begin{lem}
	\label{lem:preim_cc}
	Let $X$ be a projective variety, $Y$ a smooth projective variety and consider a birational morphism $f\colon X \rightarrow Y$. Then a subvariety $Z\subset Y$ is constant cycle if and only if $f^{-1}(Z) \subset X$ is a constant cycle subvariety. 
\end{lem}
\begin{proof}
	Consider the commutative diagram
	\begin{center}
		\begin{tikzcd}
			\textup{CH}_0(f^{-1}(Z))\arrow[d,"f_*"] \arrow[r,"\iota_*"] & \textup{CH}_0(X)\arrow[d,"f_*"]\\
			\textup{CH}_0(Z) \arrow[r,"\iota_*"] & \textup{CH}_0(Y),
		\end{tikzcd}
	\end{center}
	where $\iota$ denotes the inclusion of both subvarieties. If $X$ is smooth, the birational map $f_*$ induces an isomorphism between $\textup{CH}_0(X)$ and $\textup{CH}_0(Y)$ and the assertion follows from the commutativity of the diagram. For arbitrary $X$, use a resolution of singularities $\pi \colon \widetilde{X} \rightarrow X$ and argue as above using the above diagram and the corresponding one for $\pi$. 
\end{proof}
The following result is the main ingredient for our proof of Theorem \ref{thm:intro_wall}. Recall that the proper transform of a subvariety $Z \subset X$ under a birational map $f\colon X \dashrightarrow Y$ is defined to be $p_2(p_1^{-1}(Z))$, where $p_1$ and $p_2$ are the projections from the closure of the graph of $f$ in $X\times Y$ to $X$ respectively $Y$. A point outside the maximal open subset where the rational map $f$ is defined is called a fundamental point of $f$. 
\begin{prop}
	\label{prop:prop_cc_bir}
	Let $X$ and $Y$ be smooth projective varieties and $f\colon X \dashrightarrow Y$ a birational map. Suppose $X$ is symplectic and consider an irreducible constant cycle Lagrangian $Z \subset X$ of Kodaira dimension $\textup{kod}(Z)\ge 0$. Then, not all points of the proper transform of $Z$ can be fundamental points of the birational map $f^{-1}$. 
\end{prop}
\begin{proof}
	Assume for a contradiction that all points of the proper transform of $Z$ are fundamental points. Consider the diagram
	\begin{center}
		\begin{tikzcd}
			&\Gamma \arrow[dr,"p_2"] \arrow[dl,swap, "p_1"]&\\
			X \arrow[rr, dashed , "f"]&& Y,
		\end{tikzcd}
	\end{center}
	where $\Gamma$ denotes the closure of the graph of $f$ and $p_i$ the projections from $X\times Y$. Observe that the variety
	\[T \coloneqq p_1\left(p_2^{-1}\left(p_2\left(p_1^{-1}\left( Z\right)\right)\right)\right) \subset X\] 
	is a constant cycle subvariety containing $Z$. Indeed, preimages of constant cycle subvarieties are constant cylce by Lemma \ref{lem:preim_cc} and the image of a constant cycle subvariety under a proper map is always constant cycle.
	
	The fiber $p_2^{-1}(a) \subset \Gamma$ for a fundamental point $a$ of $f^{-1}$ is uniruled \cite[Prop.\ 1.3]{KollarMori}. Thus, the images of the fibers under the morphism $p_1$ are non-trivial uniruled varieties. If for very general $x\in Z$ there exists a fiber $p_2^{-1}(a)$ with $x \in p_1(p_2^{-1}(a))$ and a rational curve $x\in C_x \subset p_1(p_2^{-1}(a))$ such that $C_x$ is contained in $Z$, then $Z$ would be uniruled, in contradiction to kod$(Z)\ge 0$.
	
	Hence, we may assume that for very general $x \in Z$ the rational curves $C_x \subset p_1(p_2^{-1}(a))$ through $x$ are not contained inside $Z$. However, then the constant cycle subvariety $T$ strictly contains $Z$ and, thus, is of larger dimension. 
	Since the dimension of a constant cycle subvariety inside a symplectic variety is bounded by half the dimension of the ambient space, we derive a contradiction. 
\end{proof}
In a similar vein one can show that constant cycle Lagrangians $Z \subset X$ of non-negative Kodaira dimension inside a smooth symplectic variety cannot be contained inside a uniruled variety $W\subset X$. Indeed, if $Z$ would be contained in $W$, there is again through every point of $Z$ an irreducible rational curve. The very general curve can again not be contained inside $Z$. However, they must be contained in the orbit \[
O_z = \{x\in X \mid [x]=[z] \textup{ in } \CH_0(X) \}
\]
under rational equivalence of a point $z \in Z$ which is a countable union of constant cycle subvarieties \cite[Lem.\ 10.7]{VoisinBookII}. Since we consider irreducible curves $C$, there must be one irreducible component of $O_z$ which contains $C$. Thus, for a very general point $z\in Z$ not contained in any other irreducible compenent of $O_z$, the irreducible rational curve $C_z$ through $z$ must be contained in $Z$. This yields a contradiction. 
\section{Review: Stability Conditions and Moduli Spaces}
\label{sec:pre}
\subsection{Bridgeland stability conditions}
\label{subsec:Bridge_stab}

Bridgeland \cite{BridgelandSt} introduced the notion of a stability condition on a triangulated category. 

\begin{definition}
	\label{def:Bridge_stab}
	Let $\mathcal{D}$ be a triangulated category. A \textit{(full numerical) stability condition} $\sigma =(Z,\mathcal{P})$ consists of a group homomorphism $Z\colon K(\mathcal{D})\rightarrow \mathbb{C}$, called the \textit{central charge}, and for each $\phi \in \R$ a full additive subcategory $\mathcal{P}(\phi) \subset \mathcal{D}$ satisfying the following conditions
	
	\begin{itemize}
		\item If $0 \neq E \in \mathcal{P}(\phi)$, then $Z(E) \in \R_{>0} e^{i\pi \phi}$.
		\item If $\phi_1 > \phi_2$ and $E_i \in \mathcal{P}(\phi_i)$, then $\textup{Hom}(E_1, E_2)=0$.
		\item $\mathcal{P}(\phi)[1]=\mathcal{P}(\phi + 1)$.
		\item Every $0 \neq E \in \mathcal{D}$ has a categorical Harder--Narasimhan filtration.
		\item $Z$ factors through the numerical Grothendieck group $K_{\textup{num}}(\mathcal{D})$.
		\item There exists a constant $C>0$ such that for all $\phi \in \R$ and objects $E \in \mathcal{P}(\phi)$ we have $\Vert E \Vert \leq C \vert Z(E) \vert$, where we fix a norm $\Vert \_ \Vert$ on $K_{\textup{num}}(\mathcal{D}) \otimes \R$. 
	\end{itemize}
\end{definition}
Objects in $\mathcal{P}(\phi)$ are \textit{semistable} of phase $\phi$ and simple objects in $\mathcal{P}(\phi)$ are called \textit{stable}. Bridgeland constructed on the set of stability conditions Stab$(\mathcal{D})$ a generalised metric to show that it is a complex manifold \cite[Thm.\ 1.2]{BridgelandSt}. We only consider stability conditions for the bounded derived category $\Db(X)$ of a smooth projective surface $X$. 

The space of stability conditions comes naturally with two group actions. The first one is the group of exact autoequivalences Aut($\Db(X)$) acting on the left via $\psi . (Z,\mathcal{P})= (Z \circ \psi_*,\psi(\mathcal{P}))$. The universal cover $\widetilde{\textup{GL}}{}^{+}(2,\R)$ of GL$^+(2,\R)$ acts via $(Z, \mathcal{P}).(A,f) =(A^{-1}\circ Z, \mathcal{P'})$, where $\mathcal{P'}(\phi)=\mathcal{P}(f(\phi))$. 

Similiar to the case of Gieseker stability, the space of stability conditions Stab$(X)$ admits for a Mukai vector $v$ a wall and chamber decomposition \cite[Sec.\ 9]{BridgelandK3}, see also \cite[Prop.\ 2.3]{BayMacPrBiGeo}. We call a stability condition \textit{generic} (for $v$) if it does not lie on any of the walls. If $v$ is primitive and $\sigma$ is generic, then $\sigma$-stability and $\sigma$-semistability coincide. Moreover, the set of stability conditions for which a given object is stable is open in $\textup{Stab}(X)$. 

Another important property of Bridgeland stability conditions is that there is a chamber, called the large volume limit, where Bridgeland and Gieseker stability coincide \cite[Sec.\ 14]{BridgelandK3}.

\subsection{Moduli spaces of stable objects on K3 and Enriques surfaces}
\label{subsec:mod_K3_Enriq}
Let $X$ be a K3 or an Enriques surface. We denote by $H^*_{\textup{alg}}(X,\Z)$ the image of the map
\begin{align*}
v=\textup{ch}(\_) \sqrt{\textup{td}(X)} \colon K_{\textup{num}}(X) \rightarrow H^*(X,\Q).
\end{align*}	
On $H^*_{\textup{alg}}(X,\Z)$ there is the Mukai pairing defined for two vectors $v=(v_0, v_2, v_4)$ and $w=(w_0, w_2, w_4)$ as
\begin{align*}
( v,w ) \coloneqq \int_X -v_0w_4 + v_2w_2 -v_4w_0.
\end{align*}
For K3 surfaces $X$, the Mukai lattice is contained in $H^*(X,\mathbb{Z})$. For Enriques surface, since $\sqrt{\textup{td}(X)}=(1,0,\frac{1}{2})$, we have to allow rational coefficients in $H^4_{\textup{alg}}(X,\Z)$.

We first review briefly moduli spaces of stable sheaves on K3 surfaces. A Mukai vector $v=(r, c_1, s)$ is \textit{positive} if $r>0$, or $r=0$, $c_1$ is effective and $s\neq 0$, or $r=c_1=0$ and $s>0$.
\begin{thm}[Mukai, Huybrechts, O'Grady, Yoshioka]
	If $v$ is positive and primitive with $v^2>0$ and $H$ is generic, the moduli space $M_H(v)$ of stable sheaves with Mukai vector $v$ is a projective hyperkähler variety of dimension $2n = v^2+2$ deformation equivalent to $\textup{Hilb}^n(X)$. 
\end{thm}
Bridgeland constructed certain geometric stability conditions on K3 surfaces and showed that they all lie in a distinguished component $\St(X)$ of the stability manifold \cite{BridgelandK3}. For stability conditions in $\St(X)$, Bayer and Macrì showed that the above theorem also holds true for moduli spaces of stable complexes \cite{BayMacPrBiGeo}. Another key result by the two authors, which we need in our investigation of the birational type of moduli spaces, is the following, cf.\ \cite{BayMacMMP}. 

\begin{thm}
	\label{thm:BM_MMP}
	Let $X$ be a K3 surface, $v$ a primitive Mukai vector and $\sigma$ and $\tau$ generic stability conditions in $\St(X)$. Assume $v^2>0$. Then, the moduli spaces $M_{\sigma}(v)$ and $M_{\tau}(v)$ are birational. 
	
	Moreover, every smooth $K$-trivial birational model of $M_{\sigma}(v)$ appears as a moduli space $M_{\tau}(v)$ of Bridgeland stable objects for some generic $\tau \in \St(X)$. 
\end{thm}

Now we pass to Enriques surfaces $Y$. Since their canonical divisor $\omega_Y$ is 2-torsion, the moduli space for the Mukai vector $v$ decomposes into
\begin{align*}
M_{\sigma}^Y(v) \cong  M_{\sigma}^Y(v,L) \sqcup  M_{\sigma}^Y(v, L\otimes \omega_Y),
\end{align*}
where we furthermore fix the determinant line bundle $L$, respectively $L \otimes \omega_Y$. If the rank of the Mukai vector is odd, the two components are isomorphic. 

Kim \cite{KimEn, Kimr2} was the first one to study moduli spaces of stable sheaves on Enriques surfaces.
\begin{prop}[Kim]
	Given an Enriques surfaces $Y$ with its universal cover $\pi \colon \widetilde{Y} \rightarrow Y$, the morphism
	\begin{align*}
	\pi^* \colon M_H^Y(v) \rightarrow M_{\pi^*H}^{\widetilde{Y}}(\pi^*(v))
	\end{align*}
	has degree 2 and it is étale onto its image away from all points $[E]$ satisfying $E \cong E \otimes \omega_Y$. Its image is the fixed locus of the action given by the covering involution $i^* \in \textup{Aut}(M_{\pi^*H}^{\widetilde{Y}}(\pi^*(v)))$ and is a Lagrangian subvariety. 
\end{prop}
Successive results by Yoshioka \cite{yoshiokatw}, Hauzer \cite{hauzer} and finally by Nuer \cite{Nuerno, Nuerpro} have lead to a complete understanding of when the moduli spaces are non-empty. We only need
\begin{thm}
	Let $Y$ be an Enriques surface, $v \in H^*_{\textup{alg}}(Y, \mathbb{Z})$ a Mukai vector such that $\pi^*(v) \in H^*_{\textup{alg}}(\widetilde{Y}, \mathbb{Z})$ is primitive and $H$ a generic polarization. Then $M_H(v,L)$ is a non-empty smooth projective variety with torsion canonical bundle provided $v^2 \ge -1$. 
\end{thm}

The category of coherent sheaves Coh$(Y)$ on an Enriques surface $Y$ is naturally isomorphic to the category of coherent $\langle i^* \rangle$-sheaves Coh$_{i^*}(\widetilde{Y})$. This yields a natural equivalence between the bounded derived categories $\Db(Y)$ and $\Dbi(\widetilde{Y})$. In \cite{MMSISC}, the authors described stability conditions on Enriques surfaces using $i^*$-invariant stability conditions on $\St(\widetilde{Y})$ via the functors $\pi^*$ and $\pi_*$.
\begin{thm}[Macrì, Mehrotra, Stellari]
	The distinguished connected component $\St(Y)$ of the stability manifold of an Enriques surface embeds into the corresponding component of its universal cover $\St(\widetilde{Y})$. If $Y$ is generic, the components coincide. 
\end{thm}
For a generic Enriques surface $Y$ we denote by $\widetilde{\sigma} \in \St(\widetilde{Y})$ the stability condition corresponding to $\sigma \in \St(Y)$ , i.e.\ $(\pi^*)^{-1}(\widetilde{\sigma})=\sigma$.

Nuer \cite{Nuerpro} established the existence of projective coarse moduli spaces for Bridgeland stability conditions on Enriques surfaces. For primitive Mukai vector these are also smooth projective $K$-trivial varieties and the morphism
\begin{align*}
\pi^* \colon M_{\sigma}^Y(v) \rightarrow M_{\widetilde{\sigma}}^{\widetilde{Y}}(\pi^*(v))
\end{align*}
is 2:1 onto the fixed locus of the action given by the covering involution.
	\section{Moduli Spaces of stable objects on Enriques Surfaces}
\label{sec:gmes}

We first observe that the image of the moduli space of stable objects on an Enriques surface is a constant cycle Lagrangian. Shen, Yin, and Zhao \cite{ShenYinZhao} studied the group of zero-cycles on moduli spaces of stable objects on K3 surfaces. They formulated a conjecture, which was later proven by the third author and Marian \cite{MarianZhao}.
\begin{thm}[Marian, Shen, Yin, Zhao]
	Let $X$ be a K3 surface and $v\in H^*_{\text{alg}}(X,\Z)$ a primitive Mukai vector. For generic $\sigma \in \St(X)$ consider the moduli space $M_{\sigma}(v)$ of stable complexes with Mukai vector $v$. For $E,F \in M_\sigma(v)$ we have
	\begin{align*}
	[E]=[F] \textup{ in CH}_0(M_\sigma(v)) \iff c_2(E) = c_2(F) \textup{ in CH}_0(X). 
	\end{align*}
\end{thm}
By \cite{BKS}, Enriques surfaces $Y$ satisfy Bloch's conjecture. Since their Albanese variety is trivial, we get $\textup{CH}_0(Y) = \Z$. Thus, we conclude the following for all Enriques surfaces $Y$
\begin{prop}
	\label{prop:Mod_enr_CCL}
	Assume that $\pi^*(v) \in H^*_{\text{alg}}(\widetilde{Y},\Z)$ is primitive and $\sigma \in \St(\widetilde{Y})$ is generic. Then, the image of $\pi^* \colon M^Y_{\sigma}(v) \rightarrow M^{\widetilde{Y}}_{\widetilde{\sigma}}(\pi^*(v))$ is a constant cycle Lagrangian. \qed
\end{prop}
If $v=(r,c_1,\frac{s}{2}) \in H^*_{\text{alg}}(Y,\Z)$ is primitive and $r$ is even, the hypothesis of the proposition is fulfilled if and only if $2$ does not divide gcd$(r, c_1)$ \cite[Lem.\ 2.1]{Nuerno}. 

Observe that this argument does not solely work for Enriques surfaces. For example, one may take a K3 surface $X$ given as a 2:1 cover $X\rightarrow \P^2$. Similarly, one may consider K3 surfaces with a non-symplectic automorphism of finite order. The quotient also satisfies Bloch's conjecture.

The moduli space $M^Y_{\sigma}(v)$ itself is not always CH$_0$-trivial. Indeed, these moduli spaces have Kodaira dimension zero which implies that the one-dimensional moduli spaces are elliptic curves. 
However, as we will see later, moduli spaces parametrizing odd rank Mukai vectors are always $\CH_0$-trivial. 

\subsection{Wall-crossing for generic Enriques surfaces}
Recall that we have an action of Aut(D$^{\textup{b}}(\widetilde{Y}))$ on $\St(\widetilde{Y})$. By the work of Macrì, Mehrotra, and Stellari \cite{MMSISC} we know that the action of the covering involution $i^*$ is trivial on $\St(\widetilde{Y})$, since $Y$ is generic. This yields
\begin{lem}
	\label{cor:stab_i}
	The autoequivalence $i^* \in \AutD \widetilde{Y}))$ acts on the moduli space, i.e.\ for all $E \in M_{\widetilde{\sigma}}^{\widetilde{Y}}(\widetilde{v})$ we have $ i^*E \in M_{\widetilde{\sigma}}^{\widetilde{Y}}(\widetilde{v})$. In particular, if $S \in \textup{D}^{\textup{b}}(\widetilde{Y})$ is a spherical object, then $i^*S \cong S$. \qed
\end{lem}
The moduli space decomposes into $M_{\sigma}^Y(v) \cong M_{\sigma}^Y(v, L) \sqcup M_{\sigma}^Y(v, L \otimes \omega_Y)$ and we concentrate only on one component, omitting the determinant in the notation.

The above lemma enables us to show that spherical twists are equivariant functors with respect to the covering involution. Recall that a Fourier--Mukai functor $\Phi \in \textup{Aut}(\Db(\widetilde{Y}))$ is called $i^*$-equivariant, if
\[
i^* \circ \Phi \cong \Phi \circ i^*
\]
as functors. This is the case if and only if $(i\times i)^*(\mathcal{E}) \cong \mathcal{E}$, where $\mathcal{E}$ is the Fourier--Mukai kernel of $\Phi$ and $i\times i \colon \widetilde{Y} \times \widetilde{Y} \rightarrow \widetilde{Y} \times \widetilde{Y}$ is the natural map on the product. Compare the following also to \cite[Prop.\ 6.11]{Nuerpro}.
\begin{lem}
	\label{lem:Sph_Twi_equi}
	Let $S\in \Db(\widetilde{Y})$ be a spherical object and denote by $\mathcal{E}$ the kernel of the associated spherical twist. Then $(i\times i)^*\mathcal{E} \cong \mathcal{E}$. 
\end{lem}
\begin{proof}
	The sperical twist is the autoequivalence with Fourier--Mukai kernel $\mathcal{E}\in\Db(\widetilde{Y}\times \widetilde{Y})$ which is the cone
	\[
	S^{\vee} \boxtimes S \longrightarrow \mathcal{O}_{\Delta}\longrightarrow \mathcal{E} \longrightarrow S^{\vee} \boxtimes S [1]
	\]
	of the natural trace morphism \cite[Thm.\ 1.2]{SeidelThomas}. Thus, to show that $(i\times i)^*(\mathcal{E}) \cong \mathcal{E}$ it suffices to ensure that the objects $S^{\vee} \boxtimes S$ and $\mathcal{O}_{\Delta}$ as well as the trace morphism between them is invariant under the action of $(i\times i)^*$. Indeed, if the trace morphism is invariant, we can complete the distinguished triangle in the equivariant category $\Db_{(i \times i)^*}(\widetilde{Y}\times \widetilde{Y})$. 
	
	That $S^{\vee} \boxtimes S$ is invariant follows evidently from Lemma \ref{cor:stab_i} and $\mathcal{O}_{\Delta}$ is invariant under the action of $(f\times f)^*$ for every automorphism $f$ of $\widetilde{Y}$. For the trace map, one can either calculate the invariance explicitly or observe that the map is defined as the identity under the natural isomorphisms
	\[	
	\textup{Hom}_{\Db(\widetilde{Y} \times \widetilde{Y})}(S^{\vee}\boxtimes S, \mathcal{O}_{\Delta}) \cong \textup{Hom}_{\Db(\widetilde{Y})}(S^{\vee}\otimes S, \mathcal{O}_{\widetilde{Y}}) \cong \textup{Hom}_{\Db(\widetilde{Y})}(S,S). \qedhere
	\]
\end{proof}

Consider now two stability conditions $\sigma_{\scaleto{+}{4pt}}, \sigma_{\scaleto{-}{4pt}} \in \St(Y)$. Inside one chamber the moduli spaces $\MYp$ and $\MYm$ stay the same. Hence, we only need to study the relationship of these two moduli spaces for $\sigma_{\scaleto{+}{4pt}}$ and $\sigma_{\scaleto{-}{4pt}}$ in adjacent chambers. We can assume that the corresponding stability conditions $\widetilde{\sigma}_{\scaleto{+}{4pt}}$ and $\widetilde{\sigma}_{\scaleto{-}{4pt}}$ in $\St(\widetilde{Y})$ are also generic with respect to $\pi^*(v)$ and lie in adjacent chambers since $Y$ is generic. 

\begin{thm}
	\label{thm:mod_enr_wall_cross}
	Let $Y$ be a generic Enriques surface and $v \in H^*_{\textup{alg}}(Y,\Z)$ a Mukai vector such that $\pi^*(v) \in H^*_{\textup{alg}}(\widetilde{Y},\Z)$ is primitive. Then, for generic stability conditions $\sigma, \tau \in \St(Y)$ the moduli spaces $M_{\sigma}^Y(v)$ and $M_{\tau}^Y(v)$ are birationally equivalent.
\end{thm}
\begin{proof}
	To ease notation, we prove the statement at first only for odd rank Mukai vectors and describe at the end how to deduce the result in the even rank case. Recall that under this assumption $M_{\sigma}^Y(v) \hookrightarrow M_{\widetilde{\sigma}}^{\widetilde{Y}}(\pi^*(v))$ is an embedding of a constant cycle Lagrangian since we only consider one component.
	
	Consider the moduli spaces $\MYp$ and $\MYm$ in adjacent chambers and embed them inside $\MYtp$ and $\MYtm$ respectively. We know the assertion for the two moduli spaces of stable objects on the K3 surface $\widetilde{Y}$. For our purpose we need an additional property of the birational map. Observe that, since $i^*$ acts on the moduli spaces $\MYtp$ and $\MYtm$, it makes sense to ask whether the birational map is $i^*$-equivariant.
	
	The occuring birational map depends on the wall in the sense of \cite[Thm.\ 5.7]{BayMacMMP}. There are three different types. The first type induces a divisorial contraction of the moduli space. The contraction map contracts curves of stable objects that become $S$-equivalent for a stability condition on the wall. The second type is a wall inducing a flopping contraction. The remaining case is a fake wall, i.e.\ there are no curves in $\MYtp$ and $\MYtm$ that become $S$-equivalent with respect to a stability condition on the wall.
	
	We now treat each of these cases and show that the corresponding birational map is equivariant. 
	
	In case of a flopping contraction or a fake wall there either exists a common open subset whose complement has at least codimension two or the birational map is induced by the composition of spherical twists. Using Lemma \ref{lem:Sph_Twi_equi} we see that in both cases the map is equivariant. 
	
	The case of a wall inducing a divisorial contraction is divided into three subcases. If we are in the Brill--Noether case, the birational map is again defined on an open subset to be a sequence of spherical twists associated to stable spherical objects. 
	
	The second type is the Hilbert--Chow case. Here, the proof uses an isotropic vector $w \in H^*_{\textup{alg}}(\widetilde{Y},\Z)$ that satisfies $(\pi^*v, w)=1$ to identify $\Db(\widetilde{Y})$ with the bounded derived category of the fine moduli space $M$ of stable objects with class $w$. The automorphism $i$ of $\widetilde{Y}$ induces an involution $j$ on the K3 surface $M$ via pullback. This induces an autoequivalence $j^* \in \textup{Aut}(\Db(M))$. Denote by $\mathcal{E}$ the universal complex in $\Db(\widetilde{Y} \times M)$ inducing the derived equivalence. The universal property of $\mathcal{E}$ induces an isomorphism
	\[
	(i\times j)^*\mathcal{E} \cong \mathcal{E},
	\]
	where $i\times j \colon \widetilde{Y} \times M \rightarrow\widetilde{Y} \times M$ is the involution on the product. This implies that under the derived equivalence 
	\[
	\Phi_{\mathcal{E}} \colon \Db(\widetilde{Y}) \cong \Db(M)
	\]
	induced by the universal complex $\mathcal{E}$, the action of $i^*$ on $\Db(\widetilde{Y})$ is identified with the action of $j^*$ on $\Db(M)$. 
	Bayer and Macrì then show that under this identification the moduli spaces $\MYtp$ and $\MYtm$ become isomorphic via the derived autoequivalence $(\_)^{\vee}[2]$. Since this autoequivalence is obviously invariant under $j^*$, the birational map for the original moduli spaces will be invariant under the action of $i^*$. 
	
	The last occurring type is called Li--Gieseker--Uhlenbeck. We again have a moduli space $M$, such that
	\[
	\Phi_{\mathcal{E}} \colon \Db(\widetilde{Y}) \cong \Db(M,\alpha)
	\]
	for a universal complex $\mathcal{E} \in \Db(\widetilde{Y}\times M, 1\boxtimes \alpha)$, but this time we may have a Brauer class $\alpha\in H^2(\widetilde{Y},\mathcal{O}_{\widetilde{Y}}^*)_{\textup{tor}}$ of order at most 2. As $i$ is a non-symplectic involution, we again obtain an automorphism $j$ of the twisted K3 surface $(M,\alpha)$ via pullback. The involutions $i$ and $j$ are again compatible in the sense that the derived equivalence $\Phi_{\mathcal{E}}$ identifies the action of $i^*$ on $\Db(\widetilde{Y})$ with the action of $j^*$ on $\Db(M,\alpha)$. Using this identification, the birational map is given on an open subset as the functor $(\_)^{\vee} \otimes L[2]$, where $L$ is a line bundle on $M$. Since we assume $Y$ to be a generic Enriques surface, the involution $i$ acts trivially on $H^*_{\textup{alg}}(\widetilde{Y},\mathbb{Z})$. Hence, the automorphism $j$ acts trivially on the Mukai lattice of $(M,\alpha)$ and the line bundle $L$ is invariant under the induced action of $j^*$. We conclude that all the possible birational maps are equivariant.
	
	Denote by 
	\[
	f\colon \MYtp \dashrightarrow \MYtm
	\]
	the biratonal map and $U$ the biggest open subset where $f$ is an isomorphism. Since $\MYtp$ and $\MYtm$ are both hyperkähler, the set $U$ agrees with the maximal open subset, where the map $f$ is defined. The intersection $U\cap \MYp \subset \MYtp$ is non-empty by Proposition \ref{prop:prop_cc_bir}. Since $M_{\sigma_{\text{\textpm}}}^Y(v)$ can be identified with the fixed set of the involution $i^*$ and $f$ is equivariant, the restriction to the constant cycle Lagrangian $f|_{\MYp}\colon \MYp \dashrightarrow \MYm$ gives the desired birational transformation. This finishes the proof for Mukai vectors of odd rank. \smallskip
	
	If the rank is even, we can still deduce that the image of each of the two components of the moduli space $M_{\sigma_{\pm}}^Y(v)$ under $\pi^*$ intersects the open set $U$ of the map $f$. Indeed, the pullback morphism $\pi^*$ is étale onto its image and, therefore, its image is still of Kodaira dimension 0. Thus, the above argument gives us an $i^*$-equivariant Fourier--Mukai transform 
	\[
	\widetilde{\Phi} \colon \Db(\widetilde{Y}) \cong \Db(\widetilde{Y})
	\]
	inducing on an open subset a birational map between $\pi^*(\MYp)$ and $\pi^*(\MYm)$. This means that each of the two components of $\MYp$ gets mapped to precisely one of the two components of $\MYm$. 
	
	To deduce the result for the moduli space we just observe that $\pi^*$ corresponds to forgetting the equivariant structure. Indeed, the functor $\widetilde{\Phi}$ descends to a Fourier--Mukai transform 
	\[
	\Phi \colon \Db(Y) \cong \Db(Y)
	\] of the Enriques surface compatible with $\widetilde{\Phi}$ as in \cite[Thm.\ 4.5]{BridgelandMaciociaFMforquotients}. Thus, they give rise to the following commutative square
	\begin{center}
		\begin{tikzcd}
			\pi^*(\MYp)\arrow[r,dashed, "g"] & \pi^*(\MYm) \\
			\MYp \arrow[u, "\pi^*"] \arrow[r,dashed, "g'"] & \MYm \arrow[u, "\pi^*"].
		\end{tikzcd}
	\end{center}
	The fact that $g$ is birational allows us to conclude that $g'$ is birational which finishes the proof. 
\end{proof}
\subsection{Birational moduli spaces}

We go on to study the geometry of moduli spaces of stable objects on K3 and Enriques surfaces. The strategy in both cases will be the same. Suppose we are given Mukai vector $v$, which satisfies the assumption of \cite[Thm.\ 1.1]{BayMacMMP} for K3 surfaces and the ones from Theorem \ref{thm:mod_enr_wall_cross} for Enriques surfaces. We already know that for different generic stability conditions $\sigma, \tau$ the corresponding moduli spaces $M_{\sigma}(v)$ and $M_{\tau}(v)$ are birational. Applying an autoequivalence $\Phi$ of the surface induces an isomorphism 
\[
M_{\sigma}(v) \cong M_{\Phi.\sigma}(\Phi^H(v)),
\]
where $\Phi^H$ is the corresponding cohomological Fourier--Mukai functor. 
Thus, moduli spaces of stable objects with respect to a generic stability condition in the same orbit as $v$ under the action of the group of autoequivalences on the Mukai lattice are as well birational. In this way we can reduce the study of birational types of moduli spaces to the question of the orbit of the Mukai vector $v$ in the Mukai lattice. We will use this strategy throughout this section. 

Let us start with the hyperkähler manifold. 
Since the torsion-free part of the second cohomology of an Enriques surface is isometric to $U\oplus E_8(-1)$, Mukai vectors $\widetilde{v}=(r, c_1, s)$ of the form $\widetilde{v}=\pi^*(v)$ will have $c_1\in U(2) \oplus E_8(-2)$.

\begin{lem}
	\label{prop:K3_bir_Hilb}
	Let $\widetilde{Y}$ be a K3 surface with a fixed-point-free involution. Consider a primitive Mukai vector $\widetilde{v}=(r,c_1,s)$ such that $c_1\in U(2) \oplus E_8(-2) \subset \Pic(\widetilde{Y})$, $r$ odd and $\widetilde{v}^2>0$. Then, for generic $\widetilde{\sigma}$ the moduli space $M_{\widetilde{\sigma}}^{\widetilde{Y}}(\widetilde{v})$ is birationally equivalent to some Hilbert scheme of points on $\widetilde{Y}$.
\end{lem}

\begin{proof}
	By the discussion above we just have to show that the Mukai vector $\widetilde{v}$ can be transformed to $(1,0,1-n)$ using the action of autoequivalences on the Mukai lattice. By our choice of $\widetilde{v}$, it is the pullback of a primitive Mukai vector on $Y$. For Enriques surfaces, however, this has already been shown in the proof of \cite[Thm.\ 4.6]{yoshiokatw}. The autoequivalences inducing these cohomlogical Fourier--Mukai transformations lift to the K3 surface by \cite[Thm.\ 4.5]{BridgelandMaciociaFMforquotients}. 
\end{proof}
The statement also holds true for Enriques surfaces. On a generic Enriques surface there are no spherical objects. However, exceptional objects $\mathcal{E}$ give rise to \textit{weakly spherical twists} \cite[Prop.\ 6.11]{Nuerpro}. Their Fourier--Mukai kernels $\mathcal{P_{\mathcal{E}}}\in \Db(Y\times Y)$ can be defined using the distinguished triangle
\[
\mathcal{E}^{\vee}\boxtimes \mathcal{E} \oplus (\mathcal{E}\otimes \omega_Y)^{\vee} \boxtimes \mathcal{E} \otimes \omega_Y \rightarrow \mathcal{O}_{\Delta} \rightarrow \mathcal{P_{\mathcal{E}}} \rightarrow \mathcal{E}^{\vee}\boxtimes \mathcal{E} \oplus (\mathcal{E}\otimes \omega_Y)^{\vee} \boxtimes \mathcal{E} \otimes \omega_Y [1].
\]
The action of $\mathcal{P}_{\mathcal{E}}^H$ on cohomology is given by the reflection along the hyperplane perpendicular to $v(\mathcal{E})$. 
\begin{thm}
	\label{prop:Mod_Enr_bir_Hilb}
	Let $Y$ be a (not necessarily generic) Enriques surface and $v$ be a primitive Mukai vector of odd rank with $v^2>0$. Then, for generic $\sigma \in \St(Y)$, the moduli space $M_{\sigma}(v)$ is birationally equivalent to the Hilbert scheme of points $\textup{Hilb}^{n}(Y)$, where $n= (v^2+1)/2$. 
\end{thm}
\begin{proof}
	The theorem is immediate for generic Enriques surfaces using again the proof of \cite[Thm. 4.6]{yoshiokatw} and weakly-spherical twists. 
	
	Now we consider the non-generic case. Consider a family $\pi \colon \mathcal{Y}\rightarrow B$ of Enriques surfaces over a curve $B$ with central fiber a non-generic Enriques surface and very general fiber a generic Enriques surface. We construct two families. The first one is the relative Hilbert scheme 
	\[\phi \colon\textup{Hilb}_B^n(\mathcal{Y}) \rightarrow B\]
	whose fibers are the Hilbert scheme of points on $\mathcal{Y}_b$. The second is the relative moduli space of stable objects  \[
	\psi\colon M_{\mathcal{Y}/B}(v) \rightarrow B.
	\]
	Its existence is due to the recent preprint \cite{RelativeBridgelandModuliSpaces}, where the concept of stability conditions in families is established. 
	
	All fibers of both families are smooth $K$-trivial varieties. By the above we know for very general $b \in B$ that the fibers $\phi^{-1}(b)$ and $\psi^{-1}(b)$ are birationally equivalent. Fix such a birational isomorphism and the closure of its graph in the product $\textup{Hilb}^n(\mathcal{Y}_b) \times M_{\mathcal{Y}_b}(v)$. There is at least one irreducible component of the relative Hilbert scheme 
	\[\textup{Hilb}_B(\textup{\textup{Hilb}}_B^n(\mathcal{Y}) \times_B M_{\mathcal{Y}/B}(v))
	\]
	that contains uncountably many of these graphs. The universal subvariety of this component of the Hilbert scheme restricted to the special fiber establishes then the desired birational isomorphism, cf.\ \cite[Thm.\ 1]{MatsusakaMumford}. 
\end{proof}
In fact, the birational equivalence is a $K$-equivalence. Since the above construction is compatible with $\pi^*$, the inclusion of moduli spaces of stable sheaves is up to $K$-equivalence the inclusion of the Hilbert scheme of points inside the Hilbert scheme of the covering K3 surface. Hence, the constant cycle Lagrangians are already CH$_0$-trivial. 

We now want to determine the birational type in the even rank case.

\begin{prop}
	\label{prop:BirTypeEven}
	Let $Y$ be a (not necessarily generic) Enriques surface and $v$ an even rank Mukai vector such that $\pi^*(v)$ is primitive and $v^2>0$. Then, $M^Y_{\sigma}(v)$ is birational to $M^Y_{H}(v')$, where the Mukai vector $v'=(0,c_1,\frac{s}{2})$ has primitive and effective $c_1$ and $H$ is a generic polarization for $v'$.
\end{prop}
\begin{proof}
	Again, if $Y$ is generic, we want to show that we can find an autoequivalence whose action on the Mukai lattice sends $v$ to $v'$. This has been done in \cite[Sec.\ 4]{YoshiokaModuliEnriques}. The statement for non-generic Enriques surfaces can, as in the proof of Theorem \ref{prop:Mod_Enr_bir_Hilb}, be obtained via deformation theory. 
\end{proof}
The proposition enables us to apply the results of Saccà \cite[Thm.\ 3.1, Thm.\ 4.4]{SaccaEnriques}.
\begin{cor}
	Let $Y$ be a generic Enriques surface and $v \in H^*_{\textup{alg}}(Y, \Z)$ be an even rank Mukai vector such that $\pi^*(v)$ is primitive and $v^2>0$. Then, $M \coloneqq M_{\sigma}(v)$ is a smooth projective Calabi--Yau variety, i.e.
	\[
	\pushQED{\qed} 
	\omega_M \cong \mathcal{O}_M  \textup{	and	} h^{p,0}(M)=0 \textup{ for } p\neq 0, v^2+1,
	\]
	and its fundamental group is $\mathbb{Z}/2\mathbb{Z}$. 
\end{cor}
\begin{proof}
	Since the above properties are invariant under birational isomorphisms it suffices to show that these hold for one birational model. We use Proposition \ref{prop:BirTypeEven} and the fact that the missing assumption \cite[Assum.\ 2.16]{SaccaEnriques} needed for Saccà's theorems are proven in this case by Yoshioka \cite[Prop.\ 4.4]{YoshiokaModuliEnriques}.
\end{proof}
\subsection{Birational geometry of moduli spaces}
\label{subsec:MMP}
Throughout this section we assume that $v$ is a Mukai vector such that $\pi^*(v)$ is primitive and $v^2>1$. 

Bayer and Macrì constructed nef divisors $\ell_{\widetilde{\sigma}} \in \NS( M_{\widetilde{\sigma}}^{\widetilde{Y}}(\widetilde{v}))$ naturally associated to a stability condition $\widetilde{\sigma} \in \St(\widetilde{Y})$ \cite[Sec.\ 4]{BayMacPrBiGeo}. They can be defined using the composition 
\[
\St(\widetilde{Y})\xrightarrow{\mathcal{Z}} \mathcal{P}_0^+(\widetilde{Y}) \xrightarrow{I} (\pi^*(v))^{\perp} \xrightarrow{\theta_{\Wtilde}}  \NS( M_{\widetilde{\sigma}}^{\widetilde{Y}}(\pi^*(v))),
\]
where $\mathcal{Z} \colon \St(\widetilde{Y}) \rightarrow \mathcal{P}_0^+(\widetilde{Y}) \subset  H^*_{\textup{alg}}(\widetilde{Y}, \mathbb{Z})\otimes \mathbb{C}$ is the covering map which sends a stability condition $\sigma=(Z,\mathcal{P})$ to the element $\Omega_Z$ satisfying $Z(\_)=(\Omega_Z,\_)$ \cite[Thm.\ 1.1]{BridgelandK3}, $I(w)=-\textup{Im}(\frac{w}{(w,\pi^*(v))})$ and $\theta_{\widetilde{\sigma}}$ is the Mukai morphism, which is dual to the one defined in \cite[Sec.\ 8.1]{HuyLehn}.

In our arguments we use the following compatibility result.

\begin{lem}
	\label{lem:comp}
	The diagram
	\begin{center}
		\begin{tikzcd}
			\St(\widetilde{Y}) \arrow[d,"(\pi^*)^{-1}"] \arrow[r, "\mathcal{Z}"] &  \mathcal{P}^+_0(\widetilde{Y}) \arrow[d,"\pi_*"] \arrow[r, "I"] & \pi^*(v)^{\perp} \arrow[r, "\theta_{\Wtilde}"] \arrow[d, "\pi_*"] & \NS(M_{\widetilde{\sigma}}^{\widetilde{Y}}(\pi^*(v))) \arrow[d, "(\pi^*)^*"]\\
			\St(Y) \arrow[r, "\mathcal{Z}"] &  \mathcal{P}^+_0(Y) \arrow[r, "I"] & v^{\perp} \arrow[r, "\theta_{\sigma}"] & \NS(M_{\sigma}^Y(v))
		\end{tikzcd}
	\end{center}
	commutes.
\end{lem}
\begin{proof}
	The commutativity of the left square is proven in \cite[Prop.\ 3.1]{MMSISC}. The proof for the right square is analogous to the proof of \cite[Prop.\ 10.2]{Nuerpro} and the commutativity of the middle square is immediate by adjunction.
\end{proof}
For two adjacent chambers $\mathcal{C}^+, \mathcal{C}^-$ we pick again stability conditions $\sigma_{\pm} \in \mathcal{C}^{\pm}$ and a stability condition $\sigma_0$ on the wall $\mathcal{W}$ separating the two chambers. We identify $\NS(\MYp)$ and $\NS(\MYm)$ using the birational isomorphism $f$ obtained from the proof of Theorem \ref{thm:mod_enr_wall_cross}. The proof shows that these fit into the following commutative diagram
\begin{center}
	\begin{tikzcd}
		\NS(\MYtp) \arrow[d,"(\pi^*)^*"] \arrow[r, "(\widetilde{f}^{-1})^*"]& \NS(\MYtm) \arrow[d,"(\pi^*)^*"]\\
		\NS(\MYp) \arrow[r, "(f^{-1})^*"] & \NS(\MYm),
	\end{tikzcd}
\end{center}
where $\widetilde{f}$ is the birational isomorphism between $\MYtp$ and $\MYtm$ from \cite[Thm.\ 1.1(b)]{BayMacMMP}. 
This identification gives us two maps
\begin{align*}
\ell^{\pm}\colon \mathcal{C}^{\pm} \rightarrow \NS(\MYp).
\end{align*}
We now study the behaviour at the wall $\mathcal{W}$. We denote by $\mathcal{\widetilde{W}}$ the corresponding wall in $\St(\widetilde{Y})$. The stability condition $\widetilde{\sigma_0}$ produces nef and big divisor classes $\ell_{\widetilde{\sigma_0},\pm}$ on $M_{\widetilde{\sigma}_{\pm}}^{\widetilde{Y}}(\pi^*(v))$ which give rise to birational contraction morphisms 
\[
\pi_{\widetilde{\sigma}_{\pm}} \colon M_{\widetilde{\sigma}_{\pm}}^{\widetilde{Y}}(\pi^*(v)) \rightarrow \widetilde{M}_{\pm}.
\] 
For an element $d\in v^{\perp}$ satisfying $d^2\neq 0$ denote by 
\[
\rho_d \colon v^{\perp} \rightarrow v^{\perp}
\]
the involution sending a class $x\in v^{\perp}$ to $x-2\frac{\langle x,d \rangle}{d^2} d$. 
\begin{lem}
	\label{lem:ell_agree_wall}
	The maps $\ell^+$ and $\ell^-$ agree on the wall $\mathcal{W}$ when extended by continuity. 
	\begin{enumerate}[label={\upshape(\roman*)}]
		\item If $\pi_{\widetilde{\sigma}_+}$ is an isomorphism or a small contraction, then the maps $\ell^+$ and $\ell^-$ are analytic continuations of each other.
		\item If $\pi_{\widetilde{\sigma}_+}$ contracts a divisor $\widetilde{D}$, then the maps $\ell^+$ and $\ell^-$ differ in $\NS(\MYp)$ by the reflection $\rho_{[D]}$, where $D=(\pi^*)^*\widetilde{D}$. 
	\end{enumerate}
\end{lem}
\begin{proof}
	The assertion follows from the corresponding statement for the covering K3 surface \cite[Lem.\ 10.1]{BayMacMMP} and Lemma \ref{lem:comp}.
\end{proof}
Nuer \cite{Nuerpro} studied the nef divisors $\ell_{\sigma_0,\pm}$ and was able to transfer results from \cite{BayMacPrBiGeo} to moduli spaces of Enriques surfaces and showed that these are semiample as well. However, the question whether or not these divisors are big remained open.
\begin{prop}
	\label{prop:ell_big}
	The nef and semiample divisors $\ell_{\sigma_0,\pm} \in \NS (M^Y_{\sigma_{\pm}}(v))$ are big and induce birational contraction morphisms
	\[
	\pi_{\sigma_{\pm}}\colon M^Y_{\sigma_{\pm}}(v)\rightarrow M_{\pm}.
	\]
\end{prop}
\begin{proof}
	Consider the images of the moduli spaces $\pi^*(M^Y_{\sigma_{\pm}}(v)) \subset M_{\widetilde{\sigma}_{\pm}}^{\widetilde{Y}}(\pi^*(v))$. A curve $C \subset M^Y_{\sigma_{\pm}}(v)$ gets contracted to a point under the morphism associated to the linear system $| \ell_{\sigma_0,\pm} |$ if and only if its image $\pi^*(C) \subset M_{\widetilde{\sigma}_{\pm}}^{\widetilde{Y}}(\pi^*(v))$ gets contracted to a point. 
	
	There are three possible cases for the contraction on the K3 side \cite[Thm.\ 1.4]{BayMacPrBiGeo}. If the wall is a fake wall, then no curve gets contracted and the same remains true on the Enriques side. If the morphism induced by the big line bundle has an exceptional locus $B$ that is contracted, there are two possibilities. Firstly, the induced morphism is a divisorial contraction. In this case the divisor $B$ has negative square with respect to the Beauville--Bogomolov form and is uniruled \cite[Thm. 1.2]{WierzbaContractionsSymplectic}. Secondly, the codimension of the subvariety $B$ that is contracted is greater than one. In this case the general fiber of the restriction of $\pi_{\sigma_{\pm}}$ to $B$ is even rationally chain connected and, thus, $B$ is uniruled \cite[Thm.\ 1.2]{WierzbaContractionsSymplectic}.
	
	In either case the image $\pi^*(M^Y_{\sigma_{\pm}}(v))$ cannot be contained in the uniruled variety $B$ since it is a constant cycle Lagrangian of non-negative Kodaira dimension. 
\end{proof}
We now discuss minimal models of moduli spaces $M_{\sigma}(v)$ of stable objects on generic Enriques surfaces for generic $v$. 

Consider the Mukai morphism 
\begin{equation*}
\theta_{\sigma} \colon v^{\perp}\rightarrow \NS(M_{\sigma}(v)). 
\end{equation*}

In the case of a K3 surface and moduli of stable sheaves, Yoshioka \cite{YoshiokaAbelian} has proven that this morphism induces an isometry between $v^{\perp}$ with the Mukai pairing and the Néron--Severi group with the Beauville--Bogomolov form on the hyperkähler variety. This result has been generalized to Bridgeland stability conditions by Bayer and Macrì \cite[Thm.\ 6.10]{BayMacPrBiGeo}. Thus, for $a\in v^{\perp}$ to be mapped to a big and movable divisor its square has to be positive. 

Using Lemma \ref{lem:comp} and the diagram below it, to show that $\theta_{\sigma}$ is an isomorphism for moduli spaces of stable objects on Enriques surfaces one can apply Theorem \ref{prop:Mod_Enr_bir_Hilb} and Proposition \ref{prop:BirTypeEven}. Hence, it suffices to show this in the odd rank Mukai vector case for the Hilbert scheme, where one can easily check that $(\pi^*)^*$ is surjective , cf.\ \cite[Cor.\ A.4]{YoshiokaAbelian}. For an even rank Mukai vector one uses \cite[Lem.\ 12.3]{NuerYoshioka}.

However, it is unclear whether the square of a class mapping to a big and movable divisor is positive. Small evidence for this assertion is that the ample divisors constructed by Huybrechts and Lehn \cite[Thm.\ 8.1.11]{HuyLehn} and Nuer \cite[Cor.\ 12.6]{Nuerpro} all satisfy this property. 
This property is the only one missing to imitate the proof of \cite[Thm.\ 1.2(b)]{BayMacMMP}, where it is shown that all birational models of a moduli space of stable complexes on a K3 surface are again moduli spaces for a different stability condition. 

On the other hand, it is important to note that Nuer and Yoshioka \cite[Ex.\ 8.23]{NuerYoshioka} found examples of walls, which induce a small contraction, but the moduli spaces in adjacent chambers are isomorphic. It is unclear whether or not the birational model obtained by flopping the contraction is again a moduli space. 
	\bibliography{pub_bib}
	\Addresses
\end{document}